\documentclass{article}

\usepackage[cmex10]{amsmath}
\usepackage{amssymb}
\usepackage{amsthm}
\usepackage{enumerate}
\usepackage{amscd}

\theoremstyle{plain}
\newtheorem{theorem}{Theorem}
\newtheorem{lemma}[theorem]{Lemma}
\newtheorem{proposition}[theorem]{Proposition}
\newtheorem{corollary}[theorem]{Corollary}
\newtheorem*{pbA}{Problem A}
\newtheorem*{pbA*}{Problem A*}
\newtheorem{definition}[theorem]{Definition}
\newtheorem*{def0}{Definition 0}

\theoremstyle{definition}
\newtheorem{remark}[theorem]{Remark}

\usepackage{url}

\newcommand{\beq}{\begin{equation}}
\newcommand{\eeq}{\end{equation}}

\DeclareMathOperator{\Div}{Div}
\DeclareMathOperator{\Supp}{Supp}
\DeclareMathOperator{\car}{char}

\renewcommand{\div}{\textrm{div}\,}

\newcommand{\Fq}{{\mathbb{F}_q}}
\newcommand{\Z}{\mathbb{Z}}

\newcommand{\F}{\mathbb{F}}

\newcommand{\longto}{\longrightarrow}
\newcommand{\tend}{\underset{i\to\infty}{\longto}}
\newcommand{\cL}{\mathcal{L}}
\newcommand{\cO}{\mathcal{O}}
\newcommand{\cS}{\mathcal{S}}

\begin{document}

\title{$(2,1)$-separating systems beyond the probabilistic bound}

%\author{Hugues Randriambololona\\
%ENST, Paris (``Telecom ParisTech'')}

\author{Hugues Randriambololona}

\maketitle

\section{Introduction}
%\section{Introduction, definitions, and statement of results}
\label{Intro}

One of the most powerful tools to derive lower bounds in extremal
combinatorics is the so called \emph{probabilistic method} \cite{AlonSp}.
Roughly speaking, to prove the existence of an object of a given size
satisfying certain conditions, one shows that a random object of this
size (maybe after being slightly modified) has a positive probability
to satisfy these conditions.

In many problems the lower bound given by this method is conjectured
exact, at least asymptotically,
and sometimes one can prove it is indeed so.
This means that optimal solutions to such problems are rather common.
On the other hand, when the probabilistic lower bound is not
asymptotically exact, optimal solutions tend to be rare and have
some particular structure. So, from a theoretical point of view,
it is of great importance to know whether a problem belongs to one
or the other of these two classes.

%Aside from this, a slight drawback of the probabilistic method is its
%non-constructiveness. One shows the existence of a solution, but
%one cannot exhibit it explicitly
%in a reasonable time (say, polynomial in its size).
%So, it is often advantageous to complement it with constructive
%methods.
%Having a construction matching the probabilistic bound is always
%interesting,
%but of course, when possible, a construction going beyond the probabilistic
%bound is even better, since at the same time it also solves the
%alternative mentioned just above.

The problem we will be dealing with in this paper
is that of $(2,1)$-separation.
As can be seen from \cite{SagaChili}, this problem, and more generally
the theory of separating systems, has a quite long history.
While its origins could be arguably traced back to \cite{Renyi},
its first appearance, in the precise form we will be interested in,
can be found in \cite{FGU69}, motivated by a
problem in electrical engineering.
In fact the notion of separation there defined
is very natural and ubiquitous,
and several authors have introduced and studied equivalent versions,
sometimes independently,
in various contexts and in various languages
(e.g. frameproof codes, intersecting codes,
covering arrays, hash families...).
We point out the following two elegant formulations which can be
found in \cite{Korner}, the first being in terms of information theory
(dealing with binary sequences), the second in terms of extremal
combinatorics (dealing with set systems):

\begin{pbA}
How many different points can one find in the $n$-dimensional binary Hamming
space so that no three of them are on a line?
\end{pbA}
%Recall also the $n$-dimensional Hamming space is the set of length $n$
%binary sequences, the Hamming distance between two sequences being
%the number of coordinates in which they differ.)

\noindent (We say three points in a metric space are on a line
if they satisfy
the triangle inequality with equality.)

\begin{pbA*}
How many different subsets can one find in an $n$-set
so that no three $A,B,C$ of them satisfy
$A\cap B\subset C\subset A\cup B$~?
\end{pbA*}

The equivalence between these two formulations is seen by identifying
each binary sequence with its support set.

We will take the condition in Problem~A as the definition of a
$(2,1)$-separating system. Of course
it can be extended to larger alphabets. We will only use the case where
the alphabet is a (finite) field $K$:
\begin{def0}
A $(2,1)$-separating code over $K$ of length $n$ is a subset $C\subset K^n$
such that any pairwise distinct $x,y,z\in C$ satisfy
\beq
\label{triangle_strict}
d(x,z)<d(x,y)+d(y,z)
\eeq
where $d$ is the Hamming distance in $K^n$.
\end{def0}

Condition~\eqref{triangle_strict} can be rephrased as saying that
there is a coordinate $i$ such that $y_i\not\in\{x_i,z_i\}$,
or equivalently, $(x_i-y_i)(z_i-y_i)\neq0$ in $K$.
If moreover $C$ is linear, this says in turn that any
two non-zero codewords have intersecting supports.
%(proof: translate by $y$).
Such a code is then called a \emph{linear intersecting code}
\cite{LW}\cite{Miklos}\cite{CLemp}.

\medskip

There are higher notions of $(s,t)$-separating codes, and several
other variants; see the survey \cite{SagaChili}
(of notable interest in the literature, one also finds the terminology
$s$-frameproof for $(s,1)$-separating codes,
and $s$-secure-frameproof for $(s,s)$-separating codes
\cite{BS}\cite{STW}).
These can be defined either in terms of coordinates, or in terms of
metric convexity and the Hahn-Banach property, generalizing
\eqref{triangle_strict}. For a further discussion of these ideas,
and the analogy between Hamming and Euclidean spaces in this
regard, see \cite{21sep_prelim}. 
We will not need this material here, and focus on our main topic
which is as follows.

\medskip

Denote by $M(n)$ the common solution to
Problems A and A*,
%that is, the maximal size of a $(2,1)$-separating binary code of length $n$,
and define its asymptotic exponent
\beq
\rho=\limsup_{n\to\infty}\frac{\log_2M(n)}{n}.
\eeq
It is shown in \cite{Korner} that $\rho$ satisfies the inequalities
\beq
\label{2bounds}
1-\frac{1}{2}\log_23\leq\rho\leq\frac{1}{2}
\eeq
where the derivation of the lower bound
\beq
\label{0.207518}
1-\frac{1}{2}\log_23\approx 0.207518\dots
\eeq
is a typical example of use of the probabilistic method
(it also follows from the earlier works
\cite{Saga}\cite{Komlos}\cite{Miklos}\cite{CLemp},
some of them
of a more coding-theoretic nature, but still non-constructive).

The reader certainly noticed there is plenty of space
%for improvement between this lower bound and the upper bound
%(the proof of which could be, incidentally, greatly simplified,
%see \cite{}\cite{}). Indeed,
for improvement between the two bounds in \eqref{2bounds}, and indeed
the main aim of this paper will be to reduce this gap,
although by an admittedly modest quantity:

\begin{theorem}
\label{th_0.207565}
The asymptotic exponent $\rho$
satisfies the lower bound
\beq
\label{0.207565}
\rho\geq\frac{3}{50}\log_211\approx 0.207565\dots
\eeq
\end{theorem}

As we will see, the proof of this theorem is fully constructive.
And, however small the improvement
from \eqref{0.207518} to \eqref{0.207565}
might be, it is positive enough to ensure this new construction
stands beyond the probabilistic bound.
In fact, from the author's viewpoint, the tininess of this improvement
makes it even nicer, since it results from an almost miraculous
numerical coincidence.
Do Mathematics have a sense of humour?

\medskip

Our construction improves on the one of \cite{CS}, section 7.2,
while using the same \emph{concatenation} argument
(indeed it is easily
seen from \eqref{triangle_strict} that concatenating two
$(2,1)$-separating codes gives a $(2,1)$-separating code again).
The codes to be concatenated are chosen as follows:
\begin{itemize}
\item
The outer code is a linear intersecting
(hence $(2,1)$-separating) algebraic geometry code over $\F_{121}$.
\item
The inner code is any subcode of size $M=121$
out of the $128$ codewords of the binary non-linear
\emph{one-shortened Nordstrom-Robinson code} of length $n=15$.
Remark the only possible distances in this code are $0$, $6$, $8$,
and $10$, so it satisfies~\eqref{triangle_strict}.
\end{itemize}
Expressing the rate of the concatenated code as the product of the
inner and outer rates
then gives the lower bound
\beq
\label{concat121}
\rho\geq\frac{\log_2121}{15}R_{121}
\eeq
where $R_q$ denotes the asymptotic maximal achievable rate
for linear intersecting codes
over $\Fq$.

\medskip

This choice of parameters,
and especially the choice of $q=121$,
is the result of a certain trade-off that can
only be justified a posteriori.
Known lower bounds on $R_q$ get better as $q$ grows.
But on the other side, 
as in the binary case, we have the upper
bound $R_q\leq\frac{1}{2}$ (see e.g. \cite{Blackburn}),
and when $q$ and the length of the inner
code grow, the rate of the inner code becomes
limited by the asymptotic exponent $\rho$. Thus there is no hope
to get a construction of rate significantly better
than $\frac{\rho}{2}$ by this concatenation argument
if $q$ is taken too large.
In fact, candidates for the inner code have higher rate
for smaller lengths, which implies to keep $q$ of moderate size.

It turns out that the Nordstrom-Robinson code is a $(2,1)$-separating
code with exceptionally high rate
%(close to $1/2$)
%(close to one-half!)
considering its length. It is the first
of a sequence of Kerdock codes, that can be shown $(2,1)$-separating
by the very same method \cite{KrSa}, but whose parameters are of lesser
interest for concatenation.
%This choice of $\F_{121}$, instead of $\F_{128}$ (or possibly $\F_{125}$),
%is motivated by the fact that algebraic geometry provides very good
%codes over fields of \emph{square} order,
%and combines well with the intersecting support condition;
%although on the other hand this forces
%to use only $121$ of the $128$ codewords of the inner code.
Remark also that, to use the full power of algebraic geometry codes,
we need $q$ to be a square. By luck, our square $121$ is quite close
to $128$, hence restricting the Nordstrom-Robinson code to a $121$-subcode
has only marginal impact on the rate.

\medskip

In \cite{CS}
the authors remark that a linear code of relative minimum distance larger
than one-half is intersecting. Combined
with the Tsfasman-Vladut-Zink bound \cite{TVZ}
this gives
\beq
\label{TVZb}
R_q\geq\frac{1}{2}-\frac{1}{A(q)}
\eeq
where $A(q)=q^{1/2}-1$ if $q$ is a square.
Hence $R_{121}\geq 0.4$ and
$\rho\geq 0.184503$
which was the best constructive lower bound up to now.

In \cite{Xing2002} Xing gives a new criterion for an AG code to be
intersecting, that does not rely on the minimum distance of the code.
From this criterion and a (non-constructive) counting argument in the Jacobian
of the curve he deduces
\beq
\label{Xb}
R_q\geq\frac{1}{2}-\frac{1}{A(q)}+\frac{1-2\log_q(2)}{2A(q)}
\eeq
hence $R_{121}\geq 0.435546$ and 
$\rho\geq 0.200877$
which is still below \eqref{0.207518}.
However we will improve on
Xing's bound \eqref{Xb} as follows:
\begin{theorem}
Let $q$ be a prime power with $A(q)>4$.
Then the asymptotic maximal achievable rate for linear intersecting
codes over $\Fq$ satisfies
\beq
\label{HRb}
R_q\geq\frac{1}{2}-\frac{1}{2A(q)}.
%R_q\geq\frac{1}{2}-\frac{1}{A(q)}+\frac{1}{2A(q)}.
\eeq
Moreover if $q\geq 25$ is a square, then $R_q\geq\frac{1}{2}-\frac{1}{2(q^{1/2}-1)}$.
%Moreover, this lower bound is attained by an effectively constructible
%family of codes.
\end{theorem}
This new bound was first conjectured in \cite{ITW2010} and,
as noted there, it implies Theorem~1.
Indeed, it gives
$R_{121}\geq\frac{9}{20}=0.45$, hence
combined with \eqref{concat121}:
\beq
\label{re0.207565}
\rho\geq\frac{\log_2121}{15}\frac{9}{20}=\frac{3}{50}\log_211>0.207565.
\eeq
%(the improvement over \eqref{0.200877} is precisely $1/150$).
Thus the rest of this paper will be devoted to the proof of Theorem~2.
We will do so by giving an effectively constructible family of
linear intersecting codes attaining
\eqref{HRb} --- at least, provided an effectively constructible
family of curves attaining $A(q)$ is known, which will be true in our case
of interest.

\section{Algebraic geometry codes and the intersecting support property}
\label{Xingsection}

Here we recall some material from \cite{Xing2002}, and start to
develop from it.

Let $K$ be a field (in the next section we will also suppose $K$ perfect,
and actually the reader may assume $K$ is a finite field).

If $X$ is an algebraic curve
(a smooth, projective, absolutely irreducible $1$-dimensional scheme)
over $K$, of genus $g$,
and $D$ is a divisor on $X$
(in this text ``divisor'' will always mean
``$K$-rational divisor''; likewise, ``points'' will be ``$K$-points'', etc.)
we denote by
$\cL(D)=\Gamma(X,\cO(D))$ its space of global sections,
and $l(D)=\dim_{K}\cL(D)$ the dimension of the latter.

Recall $X$ admits a so called \emph{canonical divisor} $\Omega$,
which may be taken to be the divisor of any (rational)
differential form on $X$. It has degree $\deg(\Omega)=2g-2$
and dimension $l(\Omega)=g$.
The Riemann-Roch theorem asserts that
\beq
l(D)=\deg(D)+1-g+l(\Omega-D).
\eeq
In particular $l(D)\geq\deg(D)+1-g$, with equality when $\deg(D)\geq 2g-1$.

\medskip

Suppose given 
a divisor $G$ on $X$ that can be written
as a sum of distinct ($K$-)points,
each with multiplicity $1$.
Let $n=\deg(G)\leq|X(K)|$ be its degree,
and choose an ordering $P_1,\dots,P_n$ of the points in its
support, so
\beq
G=P_1+\cdots+P_n.
\eeq
Also, for each $i$, choose a local parameter $t_i$ at $P_i$.
Then, if $D$ is any divisor on $X$, the section
$t_i^{-v_{P_i}(D)}$ is a trivialization for $\cO(D)$ at $P_i$.
Restricting to the fiber, this trivialization then gives
an identification $\cO(D)|_{P_i}\simeq K$.
Now combining these identifications with the natural restriction map
\beq
\cL(D)\longto\bigoplus_{i=1}^n\cO(D)|_{P_i}
\eeq
leads to the following:

\begin{definition}
\label{defGoppa}
For any divisor $D$ on $X$,
the generalized Goppa evaluation code $C(G,D)$ is the image
of the morphism
\beq
\begin{array}{cccc}
\phi_{G,D}: & \cL(D) & \longto & K^n \\
& f & \mapsto & ((t_1^{v_1}f)(P_1),\dots,(t_n^{v_n}f)(P_n))
\end{array}
\eeq
where for each $i$ we let $v_i=v_{P_i}(D)$.
\end{definition}

The kernel of $\phi_{G,D}$ is $\cL(D-G)$. Hence, if $l(D-G)=0$,
then $\dim C(G,D)=l(D)$. This occurs for example when $\deg(D)<n$.

As noted by Xing, this construction generalizes Goppa's evaluation
codes, while allowing the supports of $G$ and $D$ to overlap
(in fact Xing's original definition also asked $D$ to be positive,
but this condition is clearly unnecessary).

A virtue of this description is
that the ordering of the $P_i$ and the choice of the $t_i$
are made once and for all, independently of $D$.
This gives some coherence in the choice of our identifications
of the fibers $\cO(D)|_{P_i}\simeq K$ as $D$ varies,
which in turn makes the system of our evaluation maps $\phi_{G,D}$
compatible, in the sense that, given two divisors $D$ and $D'$,
the following diagram is commutative:
\beq
\begin{CD}
\cL(D)\times\cL(D') @>\phi_{G,D}\times\phi_{G,D'}>> K^n\times K^n\\
@VVV @VVV\\
\cL(D+D') @>\phi_{G,D+D'}>> K^n
\end{CD}
\eeq
where the first vertical map is multiplication in
the function field $K(X)$, and the
second vertical map is termwise multiplication in $K^n$.
Indeed, both paths in the diagram send $(f,f')\in\cL(D)\times\cL(D')$
to $((t_1^{v_1+v_1'}ff')(P_1),\dots,(t_n^{v_n+v_n'}ff')(P_n))$ in $K^n$.

Said otherwise, the collection of maps $\phi_{G,D}$
define a morphism of \emph{$K$-algebras}
\beq
\phi_G:\bigoplus_{D\in\Div(X)}\cL(D)\longto K^n
\eeq
where the multiplication law in $K^n$ is termwise multiplication.

We now recall:

\begin{theorem}[Xing's criterion, \cite{Xing2002} Th.~3.5, with $s=2$]
\label{critere}
With the preceding notations, suppose $\deg(D)<n$ and
\beq
l(2D-G)=0.
\eeq
Then $C(G,D)$ has dimension $l(D)$ and is a linear (self-)intersecting code.
\end{theorem}

In fact it is possible to say slightly more.

Two linear codes $C,C'\subset K^n$ are said mutually intersecting
if any non-zero codewords $c\in C$ and $c'\in C'$
have intersecting supports. Then:

\begin{proposition}
\label{Xinggen}
Suppose $D,D'$ are divisors on $X$ with $\deg(D)<n$, $\deg(D')<n$, and
\beq
l(D+D'-G)=0.
\eeq
Then $C=C(G,D)$ and $C'=C(G,D')$ have dimension $l(D)\geq\deg(D)+1-g$ and
$l(D')\geq\deg(D')+1-g$
respectively, and are mutually intersecting.
\end{proposition}
\begin{proof}
Let $c\in C$ and $c'\in C'$ and suppose the termwise product $cc'$
is zero in $K^n$. Write $c=\phi_{G,D}(f)$ and $c'=\phi_{G,D'}(f')$
for $f\in\cL(D)$ and $f'\in\cL(D')$.
Then $0=cc'=\phi_{G,D+D'}(ff')$
so $ff'\in\ker\phi_{G,D+D'}=\cL(D+D'-G)=\{0\}$, hence $f=0$ or $f'=0$,
that is $c=0$ or $c'=0$. This proves the intersection property,
and then the lower bound on the dimensions follows from Riemann-Roch.
\end{proof}

Remark that this proposition includes Theorem~\ref{critere} as a particular
case (namely when $D=D'$). In fact the proof given here is essentially the
same as Xing's.

We will use this variant of
Xing's criterion to give a lower bound on the rates of pairs of
mutually intersecting codes.
While easier, the
proof of this result will serve as a model for the proof of Theorem~2
in the last
section; and it is also certainly of independent interest.

\begin{lemma}
\label{l(A+P)=l(A)}
Let $X$ be a curve over $K$ of genus $g$,
and let $A$ be a divisor on $X$ with $\deg(A)\leq g-2$
and
\beq
%\label{special}
l(A)=0.
\eeq
Then for all points $P\in X(K)$ except perhaps for at most $g$ of them,
we have
\beq
\label{l(A+P)=0}
l(A+P)=0.
\eeq
\end{lemma}
\begin{proof}[Proof (adapted from \cite{Stichtenoth}, ch.~I, claim (6.8))]
By contradiction, suppose there are $g+1$ distinct
points $P_1,\dots,P_{g+1}\in X(K)$
for which \eqref{l(A+P)=0} fails, hence for each $1\leq i\leq g+1$
we can find a function $f_i\in\cL(A+P_i)\setminus\cL(A)$.
Let also
\beq
A'=A+P_1+\cdots+P_{g+1}.
\eeq
Then we also have $f_i\in\cL(A')\setminus\cL(A'-P_i)$,
which means that the quotient space $\cL(A')/\cL(A'-P_i)$ has
dimension $1$ and admits $f_i$ as a generator. On the other hand,
for $j\neq i$, we have $f_i\in\cL(A'-P_j)$, hence $f_i$ maps
to $0$ in $\cL(A')/\cL(A'-P_j)$.
This implies that the natural map
\beq
\cL(A')\longto\bigoplus_{i=1}^{g+1}\cL(A')/\cL(A'-P_i)
\eeq
is onto, hence by the rank theorem,
\beq
l(A')\geq g+1.
\eeq
But at the same time
\beq
\deg(A')=\deg(A)+g+1\leq 2g-1
\eeq
which contradicts Riemann-Roch.
\end{proof}

\begin{proposition}
\label{constrmutint}
Suppose $n>g$.
Let $m$ be an integer with $g\leq m<n$
and let $D$ be a divisor on $X$ of degree $\deg(D)=m$.
Then there exists a divisor $D'$ on $X$, with
$\deg(D')=n+g-1-m$,
such that
$C=C(G,D)$ and $C'=C(G,D')$ 
have dimension
\beq
\dim C\geq m+1-g
\eeq
and
\beq
\dim C'\geq n-m
\eeq
respectively, and are mutually intersecting.
\end{proposition}
\noindent (Note the hypothesis $n>g$ implies $|X(K)|>g$, hence in some way
``$X$ has many rational points''.)
\begin{proof}
For $0\leq i\leq g$ we construct divisors $D'_i$
such that
\beq
\label{conditions}
\deg(D'_i)=n+i-1-m\qquad\textrm{and}\qquad l(D+D'_i-G)=0
\eeq
iteratively
as follows:
\begin{itemize}
\item Start with any divisor $D'_0$ 
of degree $n-1-m$,
hence $\deg(D+D'_0-G)<0$ and $l(D+D'_0-G)=0$ as asked.
\item Suppose up to some $i<g$, we have a divisor $D'_i$
satisfying \eqref{conditions}.
The divisor $A=D+D'_i-G$ has then degree $\deg(A)=i-1$
and $l(A)=0$,
and since $|X(K)|>g$,
we can apply Lemma \ref{l(A+P)=l(A)}
to find $P$ such that $l(A+P)=0$.
Then we set $D'_{i+1}=D'_i+P$, so $D'_{i+1}$ satisfies \eqref{conditions}.
\item This ends when $i=g$, and we set $D'=D'_g$.
\end{itemize}
With this choice of $D'$, the conditions in Proposition~\ref{Xinggen}
are satisfied, hence $C$ and $C'$ are mutually intersecting.
Moreover, $\dim C'=l(D')\geq\deg(D')+1-g=n-m$ as claimed.
\end{proof}

Let $q$ be a prime power.
Say that a sequence of curves $X_i$ over the finite field $\Fq$
form an $\infty$-sequence if the genus $g_i$ 
of $X_i$ tends to infinity as
$i$ goes to infinity.

Let $A(q)$ be the \emph{largest} real number such that there exists
an $\infty$-sequence of curves $X_i$ over $\Fq$
with
\beq
\frac{|X_i(\Fq)|}{g_i}\tend A(q).
\eeq
An $\infty$-sequence of curves 
for which this limit is attained
is then said \emph{optimal}.

It is known \cite{DV} that $A(q)\leq q^{1/2}-1$,
with equality when $q$ is a square \cite{Ihara}\cite{TVZ}.

\begin{corollary}
Suppose $A(q)>1$.
Let $r$ and $r'$ be two positive real numbers such that
\beq
r+r'\leq 1-\frac{1}{A(q)}.
\eeq
Then there exists a sequence of pairs of mutually intersecting
codes $(C,C')$ over $\Fq$, of length going to infinity,
and of rates at least asymptotically $(r,r')$.
\end{corollary}
\begin{proof}
Let $X_i$ be curves forming an optimal sequence over $\Fq$.
Let $G_i=\sum_{P\in X_i(\Fq)}P$
be the sum of all rational points in $X_i$. 
Since $A(q)>1$, one has $n_i=\deg(G_i)=|X_i(\Fq)|>g_i$
for $i$ big enough.
Let $m_i$ be a sequence of integers such that
$\frac{m_i}{n_i}\longto r+\frac{1}{A(q)}$ as $i$ goes to infinity
(hence $g_i\leq m_i<n_i$ if $i$ is big enough).
Let $D_i$ be an arbitrary divisor of degree $m_i$ on $X_i$,
and apply the preceding Proposition~\ref{constrmutint}.
This gives mutually intersecting codes $(C_i,C_i')$, where
the rate of $C_i$ is at least
\beq
\frac{m_i+1-g_i}{n_i}\tend r+\frac{1}{A(q)}-\frac{1}{A(q)}=r
\eeq
and the rate of $C_i'$ is at least
\beq
\frac{n_i-m_i}{n_i}\tend 1-r-\frac{1}{A(q)}\geq r'
\eeq
as claimed.
\end{proof}

Remark that for $r=r'$,
this last corollary gives a family of mutually intersecting codes $(C,C')$
of asymptotic rate $\frac{1}{2}-\frac{1}{2A(q)}$.
This can be seen as a weak version of Theorem~2, which asserts
that this can be done with $C=C'$ (but with more restrictive conditions
on $q$).

\section{The construction}
\label{laconstruction}

From now on $K$ is assumed to be a \emph{perfect} field.

The main technical tool in the proof of Theorem~2 will be the following
``higher version'' of Lemma~\ref{l(A+P)=l(A)}:

\begin{lemma}
\label{l(A+2P)=l(A)}
Let $X$ be a curve over $K$ of genus $g$,
and let $A$ be a divisor on $X$ with $\deg(A)\leq g-3$
and
\beq
%\label{special}
l(A)=0.
\eeq
Then for all points $P\in X(K)$ except perhaps for at most $4g$ of them,
we have
\beq
\label{l(A+2P)=0}
l(A+2P)=0.
\eeq
%\label{Sgros}
%\label{B>2}
%\label{Tgros}
\end{lemma}
\begin{proof}
We can assume $|X(K)|>4g\geq g$, otherwise there is nothing to prove.
Then, thanks to Lemma~\ref{l(A+P)=l(A)}, successively adding points
to $A$, we can find a divisor $A'\geq A$ with $\deg(A')=g-3$ and
$l(A')=0$. Then for any $P\in X(K)$ with $l(A+2P)>0$, we also
have $l(A'+2P)>0$. So we can replace $A$ with~$A'$, that is,
it suffices to prove
Lemma~\ref{l(A+2P)=l(A)} with $\deg(A)=g-3$.
In turn, by Riemann-Roch,
setting $B=\Omega-A$ where $\Omega$ is a canonical divisor on~$X$,
this is equivalent to the following statement:

\emph{If $B$ is a divisor on $X$ with $\deg(B)=g+1$ and $l(B)=2$,
then there are at most $4g$ points $P\in X(K)$ with $l(B-2P)>0$.}

Replacing $B$ by a linearly equivalent divisor, we can suppose $B\geq0$.
Let then $\{1,f\}$ be a basis of $\cL(B)$.
We will conclude by a degree argument on the differential form $df$.

First we claim that $df$ is non-zero. If $\car K=0$ this is true
because $f$ is non-constant. If $\car K=p>0$ then, since $K$ is
assumed perfect, $df=0$ means $f=h^p$ for some $h\in K(X)$.
But then $h\in\cL(\frac{1}{p}B)\subset\cL(B)$ and
$\{1,h,f\}$ are linearly independent in $\cL(B)$, contradicting
our hypothesis $l(B)=2$.

Let $\mathcal{S}=\{\,P\in X(K)\:|\;l(B-2P)>0\:\}$.
We have to show $|\mathcal{S}|\leq 4g$.

Now if $P$ is a closed point (of arbitrary degree) in $X$, we are
in one of these four mutually exclusive situations:

\begin{enumerate}[(i)]
\item $P\not\in \mathcal{S}\cup\Supp(B)$.
Then $v_P(f)\geq0$, and
$v_P(df)\geq 0$.
\item $P\in\Supp(B)\setminus\mathcal{S}$.
Then $v_P(B)\geq1$, and $v_P(df)\geq v_P(f)-1\geq -v_P(B)-1$ hence
\beq
v_P(df)\geq -2v_P(B).
\eeq
\item $P\in\mathcal{S}\setminus\Supp(B)$.
Consider the inclusions $\cL(B-2P)\subset\cL(B-P)\subset\cL(B)$.
By hypothesis $l(B)=2$ and $l(B-2P)>0$, and
since $1\in\cL(B)\setminus\cL(B-P)$, necessarily $\cL(B-P)=\cL(B-2P)$.

Now let $\alpha=f(P)$. Then $f-\alpha\in\cL(B-P)=\cL(B-2P)$,
so $v_P(f-\alpha)\geq2$, hence $v_P(d(f-\alpha))\geq1$.
But since $df=d(f-\alpha)$ we conclude:
\beq
v_P(df)\geq 1.
\eeq
\item $P\in\mathcal{S}\cap\Supp(B)$.
By hypothesis $v_P(f)\geq -v_P(B)$
and $v_P(B)\geq1$.
We claim it is impossible to have simultaneously $v_P(f)=-v_P(B)$
and $v_P(B)=1$.

For if it were the case, then $f\in\cL(B)\setminus\cL(B-P)$
and $1\in\cL(B-P)\setminus\cL(B-2P)$, so all inclusions
$\cL(B-2P)\subset\cL(B-P)\subset\cL(B)$ would be strict, contradicting
$l(B)=2$ and $l(B-2P)>0$.

So one (at least) of the inequalities $v_P(f)\geq -v_P(B)$
and $v_P(B)\geq1$ is strict. Thus
$v_P(df)\geq v_P(f)-1>-2v_P(B)$, that is:
\beq
v_P(df)\geq -2v_P(B)+1.
\eeq
\end{enumerate}
Now summing these inequalities we find
\beq
2g-2=\deg(\div df)=\sum_Pv_P(df)\deg(P)\geq -2\deg(B)+|\cS|
\eeq
and since $\deg(B)=g+1$ this gives $|\cS|\leq 4g$ as claimed.
\end{proof}

%\begin{remark}
%The more geometrically inclined reader certainly noticed the proof
%could be reformulated as follows. 
%Let $B_0$ be the base locus of $B$, and let $b=\deg B$, $b_0=\deg B_0$.
%The complete linear system associated with $B$, or with $B-B_0$,
%gives rise to a morphism $\phi:X\longto\mathbb{P}^{l(B)-1}$,
%of degree $b-b_0$.
%The choice of a $2$-dimensional subspace of $\cL(B)$ (generated by $x,y$)
%is the same thing as the choice of a codimension~$2$ linear subspace
%$Z\subset\mathbb{P}^{l(B)-1}$, which gives rise to a linear projection
%$\psi:\mathbb{P}^{l(B)-1}\setminus Z\longto\PP$.
%The composite $\psi\circ\phi:X\setminus\phi^{-1}(Z)\longto\PP$ extends
%to the whole of $X$, giving the morphism
%$\widetilde{f}:X\longto\PP$, separable if $Z$ is suitably chosen.
%Now for $P\not\in B_0$, the condition $l(B-2P)>l(B)-2$
%means that $\phi$ has vanishing differential at $P$
%(this is part of the criterion for immersion of curves in projective
%space, as in \cite{Hartshorne} Prop.~IV.3.1),
%hence for $P\not\in(B_0\cup\phi^{-1}(Z))$, $\psi\circ\phi$ is ramified
%at $P$. If $H$ is any hyperplane containing $Z$,
%then $|\phi^{-1}(Z)|\leq|\phi^{-1}(H)|\leq b-b_0$,
%hence $\widetilde{f}$ is ramified at least at $|\cT|-b_0-(b-b_0)$ points,
%and the conclusion follows again from Hurwitz's formula.
%
%In turn, all this could then be rephrased and generalized
%in the broader context
%of Weierstrass gap sequences, and Pl\"ucker and Brill-Segre formulas,
%as explained in the second part of Remark~\ref{rem_finale} below.
%\end{remark}

\begin{proposition}
\label{Main_Prop}
Let $X$ be a curve over $K$ of genus $g$, and suppose
\beq
\label{X>8g-2}
|X(K)|>4g.
\eeq
Let $G$ be a divisor on $X$, of degree $n=\deg(G)\in\Z$.
Then there exists a divisor $D$ on $X$ of degree
$\deg(D)=\left\lfloor\frac{n+g-1}{2}\right\rfloor$
(or equivalently: $g-2\leq\deg(2D-G)<g$), such that
\beq
l(2D-G)=0.
\eeq
\end{proposition}
\begin{proof}
%If $g=0$ this is obvious, so we will assume $g\geq 1$.
For $0\leq i\leq N=\left\lfloor\frac{n+g-1}{2}\right\rfloor-\left\lfloor\frac{n-1}{2}\right\rfloor$ we construct divisors $D_i$
such that
\beq
\label{conditions2}
\deg(D_i)=i+\left\lfloor\frac{n-1}{2}\right\rfloor\qquad\textrm{and}\qquad l(2D_i-G)=0
\eeq
iteratively
as follows:
\begin{itemize}
\item Start with any
%and our assumption $g\geq 1$.}
divisor $D_0$ 
of degree $\left\lfloor\frac{n-1}{2}\right\rfloor$,
hence $\deg(2D_0-G)<0$ and $l(2D_0-G)=0$ as asked.
For example take $P_0\in X(K)$ and set
$D_0=\left\lfloor\frac{n-1}{2}\right\rfloor P_0$ ---
remark that $X(K)$ is non-empty, because of \eqref{X>8g-2}.
\item Suppose up to some $i<N$, we have a divisor $D_i$
satisfying \eqref{conditions2}.
The divisor $A=2D_i-G$ then satisfies $-2\leq\deg(A)<g-2$
and $l(A)=0$, so by~\eqref{X>8g-2} and
Lemma \ref{l(A+2P)=l(A)} we can find $P\in X(K)$ such that $l(A+2P)=0$.
Then we set $D_{i+1}=D_i+P$, and $D_{i+1}$ satisfies \eqref{conditions2}.
\item This ends when $i=N$, and we can set $D=D_N$.
\end{itemize}
\end{proof}
Remark that the construction given in the proof involves
roughly $g/2$ iterations, and each step requires testing
at most $4g+1$ points.
So, as soon as a curve of genus $g$, 
as well as sufficiently many of its rational points, and the various
Riemann-Roch spaces $\cL(A)$, can be computed in time polynomial in $g$,
then the overall construction will be polynomial in $g$.

\begin{corollary}
Let $X$ be a curve over $K$ of genus $g$,
such that $|X(K)|>4g$.
Let $n$ be an integer such that
$g<n\leq|X(K)|$. Then there exists a linear intersecting code $C$ over $K$,
of length $n$ and dimension
\beq
\dim C\,\geq\,\left\lfloor\frac{n+g-1}{2}\right\rfloor+1-g\;\geq\:\frac{n-g}{2}.
\eeq
\end{corollary}
\begin{proof}
Let $G=P_1+\cdots+P_n$ for pairwise distinct $P_i\in X(K)$.
The proposition gives a divisor $D$ on $X$ of degree
$\deg(D)=\left\lfloor\frac{n+g-1}{2}\right\rfloor<n$
with $l(2D-G)=0$.
The conclusion then follows from Theorem~\ref{critere},
with $C=C(G,D)$.
\end{proof}

We can now proceed with:

\begin{proof}[Proof of Theorem~2]
Let $X_i$ be curves forming an optimal sequence over $\Fq$,
let $g_i$ be the genus of $X_i$, and let $G_i$ be the sum
of all points in $X_i(\Fq)$, so $n_i=\deg(G_i)=|X_i(\Fq)|$.
By definition we have $g_i\to\infty$ and $n_i/g_i\to A(q)>4$,
so $n_i>4g_i$ if $i$ is big enough. The preceding corollary then
gives a linear intersecting code $C_i$ over $\Fq$ of rate at least
\beq
\frac{1-g_i/n_i}{2}\tend\frac{1}{2}-\frac{1}{2A(q)}
\eeq
as asked.

If $q\geq25$ is a square, then $A(q)=q^{1/2}-1$, and the conclusion
follows, except perhaps for $q=25$, $A(q)=4$.
But in this last particular case, we know that the sequence of
modular curves $X_0(11\ell)$, for $\ell\geq13$ prime,
has genus $g_\ell=\ell$ and number of points
$|X_0(11\ell)(\F_{25})|\geq 4\ell+4>4g_\ell$,
and we conclude in the same way.
\end{proof}
\noindent
As regards constructiveness issues in this last proof, 
note that when $q$ is a square, such optimal sequences are known
explicitly (see for example \cite{GS}),
and all computations in the proposition can be made
in polynomial time,
hence the overall construction can be made in polynomial time
(although perhaps with constants and exponents too big to be
really useful in practice).

\begin{remark}
\label{rem_finale}
We finish by noting two possible improvements on Theorem~2.
\begin{enumerate}[(i)]
\item In fact the hypothesis $A(q)>4$
(or $q=25$, $A(q)=4$) in Theorem~2 is not optimal.
This constant~$4$ comes from Lemma~\ref{l(A+2P)=l(A)},
and it turns out that the estimation in this lemma
(as well as the one in Lemma~\ref{l(A+P)=l(A)}, by the way)
can be slightly
improved, as done in~\cite{rang_mult}
(the proof is more technically involved).

From this stronger version of the lemma one can show that
the conclusion in Theorem~2 holds already
when $A(q)\geq 4-\frac{12q^2-4}{q^4+2q^2-1}$
(see~\cite{rang_mult} for more details).

Clearly this improvement is small, not to say unimpressive,
and for the application to Theorem~1, we only need the case
$q=121$, $A(q)=10$, so we can leave such refinements apart.
Nevertheless, further relaxing of the condition on $q$ in
Theorem~2 could have interest by itself.
\item Theorem~2 is concerned only in improving the case $s=2$
of Xing's bound \cite{Xing2002} (we will keep his notations,
so our $R_q$ becomes $R_q(2)$),
since this is all we need
for Theorem~1 again. However,
following \cite{ITW2010},
it is natural to conjecture that,
for any $s$, and maybe under suitable conditions on $q$,
\beq
\label{bgen}
R_q(s)\geq\frac{1}{s}-\frac{1}{A(q)}+\frac{1}{sA(q)}.
\eeq
If one tries to prove \eqref{bgen}
with a method similar to the one given here,
%for $s=2$,
one will construct inductively some
divisors $A$ of controlled degree and dimension $l(A)=0$, and
the main point will be to show that, given
sufficiently many points, there is one of them, say $P$, such that
$l(A+sP)=0$
(of which Lemma~\ref{l(A+P)=l(A)} is the case $s=1$
and Lemma~\ref{l(A+2P)=l(A)} the case $s=2$).
Equivalently (see e.g. \cite{Laksov} or \cite{SV})
one has to prove that $B=\Omega-A$ has order sequence at $P$
starting with $\epsilon_0(P)=0,\dots,\epsilon_{s-1}(P)=s-1$.
A necessary condition for this to be possible, is that $B$ has \emph{generic}
order sequence starting with $\epsilon_0=0,\dots,\epsilon_{s-1}=s-1$.
If this holds, the existence of $P$ can be derived from a Pl\"ucker formula
(\cite{Laksov}, Theorem~9).

For $s=2$, it is known
that any complete linear system has generic order sequence
starting with $\epsilon_0=0$ and $\epsilon_1=1$.
In our situation, this is equivalent to the non-vanishing of $df$
established during the proof of Lemma~\ref{l(A+2P)=l(A)} --- and
then the proof of Lemma~\ref{l(A+2P)=l(A)} proceeds with a variant
of the Pl\"ucker formula suitable for our particular case (classically,
this relies on Wronskians; in the proof given here, the Wronskian
is just $df$).

Unfortunately, for $s\geq3$, not all divisors $B$ have
generic order sequence starting with
$\epsilon_0=0,\dots,\epsilon_{s-1}=s-1$.
While this is known to hold for
``most'' divisors \cite{Neeman},
it might be difficult to ensure that it is so for
the particular divisors constructed in an inductive procedure
such as ours.
\end{enumerate}
\end{remark}

\end{document}